\documentclass[11pt,a4paper,reqno]{amsart}
\usepackage[ansinew]{inputenc}
\usepackage[T1]{fontenc}
\usepackage[english]{babel}
\usepackage{amsthm}
\usepackage{amsfonts}     	
\usepackage{amsmath}
\usepackage{amssymb}
\usepackage{esint}
\usepackage{hyperref}

\def\vint_#1{\mathchoice%
           {\mathop{\kern 0.2em\vrule width 0.6em height 0.69678ex depth -0.58065ex
                   \kern -0.8em \intop}\nolimits_{\kern -0.4em#1}}%
           {\mathop{\kern 0.1em\vrule width 0.5em height 0.69678ex depth -0.60387ex
                   \kern -0.6em \intop}\nolimits_{#1}}%
           {\mathop{\kern 0.1em\vrule width 0.5em height 0.69678ex
               depth -0.60387ex
                   \kern -0.6em \intop}\nolimits_{#1}}%
           {\mathop{\kern 0.1em\vrule width 0.5em height 0.69678ex depth -0.60387ex
                   \kern -0.6em \intop}\nolimits_{#1}}}
\def\vintslides_#1{\mathchoice%
           {\mathop{\kern 0.1em\vrule width 0.5em height 0.697ex depth -0.581ex
                   \kern -0.6em \intop}\nolimits_{\kern -0.4em#1}}%
           {\mathop{\kern 0.1em\vrule width 0.3em height 0.697ex depth -0.604ex
                   \kern -0.4em \intop}\nolimits_{#1}}%
           {\mathop{\kern 0.1em\vrule width 0.3em height 0.697ex depth -0.604ex
                   \kern -0.4em \intop}\nolimits_{#1}}%
           {\mathop{\kern 0.1em\vrule width 0.3em height 0.697ex depth -0.604ex
                   \kern -0.4em \intop}\nolimits_{#1}}}

\newcommand{\aveint}[2]{\mathchoice%
           {\mathop{\kern 0.2em\vrule width 0.6em height 0.69678ex depth -0.58065ex
                   \kern -0.8em \intop}\nolimits_{\kern -0.45em#1}^{#2}}%
           {\mathop{\kern 0.1em\vrule width 0.5em height 0.69678ex depth -0.60387ex
                   \kern -0.6em \intop}\nolimits_{#1}^{#2}}%
           {\mathop{\kern 0.1em\vrule width 0.5em height 0.69678ex depth -0.60387ex
                   \kern -0.6em \intop}\nolimits_{#1}^{#2}}%
           {\mathop{\kern 0.1em\vrule width 0.5em height 0.69678ex depth -0.60387ex
                   \kern -0.6em \intop}\nolimits_{#1}^{#2}}}

\newcommand{\R}{\mathbb{R}}

\newcommand{\N}{\mathbb{N}}

\newcommand{\diam}{\operatorname{diam}}

\renewcommand{\epsilon}{\varepsilon}

\theoremstyle{plain}

\newtheorem{lause}{Theorem}[section]
\newtheorem{lem}[lause]{Lemma}

\newtheorem{maar}[lause]{Definition}

\theoremstyle{remark}

\numberwithin{equation}{section}
\begin{document}

\title[Regularity of game value functions]{Local regularity results for value functions of tug-of-war with noise and running payoff}

\author[Eero Ruosteenoja]{Eero Ruosteenoja}
\address{Department of Mathematics and Statistics, University of
Jyvaskyla, PO~Box~35, FI-40014 Jyvaskyla, Finland}
\email{eero.ruosteenoja@jyu.fi}
\date{\today}

\setlength{\parindent}{15pt}
\begin{abstract}
We prove local Lipschitz continuity and Harnack's inequality for value functions of the stochastic game tug-of-war with noise and running payoff. As a consequence, we obtain game-theoretic proofs for the same regularity properties for viscosity solutions of the inhomogeneous $p$-Laplace equation when $p>2$.  
\end{abstract}
\maketitle
\section{introduction} 
Max and Minnie play a zero-sum stochastic game as follows. Fix $\epsilon>0$. First a token is placed at $x_0\in \Omega$, where $\Omega\subset \R^n$ is a bounded domain. With probability $\alpha\in (0,1)$ they flip a fair coin and the winner can move the token anywhere in an open ball $B_\epsilon (x_0)$. With probability $\beta=1-\alpha$ the token moves to a random point in $B_\epsilon(x_0)$. The game continues until the token hits the boundary of $\Omega$ for the first time in, say $x_\tau$. Then Minnie pays Max a \emph{total payoff} 
\[
F(x_\tau)+\epsilon^2\sum^{\tau-1}_{j=0}f(x_j),
\] 
where $F$ is a bounded \emph{final payoff} on the boundary and $f$ a positive, bounded \emph{running payoff} in $\Omega$. Since $\beta>0$, the game ends almost surely in a finite time. Max tries to maximize the total payoff and Minnie tries to minimize it.

For given payoffs and $\epsilon>0$, the game has a value $u_\epsilon$, which is locally Lipschitz continuous up to the scale $\epsilon$. To be more precise, we show in Theorem \ref{Lemma 1} that if $B_{6R}(a)\subset \Omega$ and $\epsilon<r\leq R$, then
\[
\text{osc}(u_\epsilon, B_r(a))\leq C\frac{r}{R}\left[\text{osc}(u_\epsilon,B_{6R}(a))+\text{osc}(f,B_{6R}(a))\right],
\]
where $C>0$ depends only on $p$ and $n$. In Theorem \ref{päälause} we show that if $B_{30r}(a)\subset \Omega$, the value function $u_\epsilon$ satisfies Harnack's inequality
\[
\sup_{B_r(a)}u_\epsilon\leq K (\inf_{B_r(a)}u_\epsilon+\sup_{\Omega} f),
\]
where $K=K(p,n)>0$. In the proofs of Theorems \ref{Lemma 1} and \ref{päälause} key ideas are related to controlling the expected cumulative effect of running payoff during the game under proper strategies. 

According to Lemmas \ref{DPP1} and \ref{DPP2}, the value functions satisfy
\begin{equation}\label{dynaamista}
u_\epsilon(x)=\frac{\alpha}{2}\left\{\sup_{B_\epsilon(x)}u_\epsilon+\inf_{B_\epsilon(x)}u_\epsilon\right\}+\beta\vint_{B_\epsilon(x)}u_\epsilon\ dy+\epsilon^2f(x)
\end{equation}
for all $x\in \Omega$. Choosing the probabilities $\alpha$ and $\beta$ properly, this \emph{dynamic programming principle} (hereafter DPP) gives a connection to viscosity solutions of the inhomogeneous $p$-Laplace equation
\begin{equation}\label{vissi1}
-\frac{1}{p}\left|\nabla u\right|^{2-p}\text{div}(\left|\nabla u\right|^{p-2}\nabla u)=:-\Delta^N_p u=f,
\end{equation}
where $p>2$. Let $f>0$ be continuous and bounded. If $u$ is a viscosity solution to \eqref{vissi1} in $\Omega$ with some continuous and bounded boundary values, by Lemma \ref{raja} there is a sequence $(u_\epsilon)$ of value functions converging locally uniformly to $u$. By Theorems \ref{Lemma 1} and \ref{päälause}, the function $u$ is locally Lipschitz continuous and  satisfies Harnack's inequality. To the best of our knowledge, Lipschitz estimate is unavailable in the literature in the case $2<p\leq n$. In the case $p>n\geq 2$ similar estimates were proven in \cite{Charro} by using PDE methods. Harnack's inequality also follows by utilizing PDE methods described e.g.\ in \cite{Caffarelli,Gilbarg}. For recent advances, see e.g.\ \cite{kuusi}.
     
Tug-of-war games were first introduced by Peres, Schramm, Sheffield and Wilson in \cite{Peres} and by Peres and Sheffield in \cite{PS}. Various versions of the game have connections to the theory of nonlinear PDEs. For example, value functions of the tug-of-war approximate infinity harmonic functions and value functions of the tug-of-war with noise approximate $p$-harmonic functions. Game-theoretic arguments have generated many new, intuitive proofs for uniqueness and regularity properties of infinity harmonic and $p$-harmonic functions. See e.g.\ \cite{Antunovic,Armfinite,Arm,Bjorland1,Bjorland2,Lopez,Harnack,rossi}. For existence of viscosity solutions to certain parabolic equations, see e.g.\ \cite{para1,para2}. For different versions of DPP, see e.g.\ \cite{Hart}.

In Section \ref{sec 2} we define the game, show that it has a value which satisfies DPP \eqref{dynaamista}, and give tools to estimate the value functions under different strategies and payoffs. In Section \ref{sec 3} we prove lemmas regarding expected stopping times under specific strategies, local comparison of value functions and control of infimum. In Section \ref{sec 4} we prove Theorems \ref{Lemma 1} and \ref{päälause}. In Section 5 we discuss the connection to the inhomogeneous $p$-Laplace equation.

\section{Background of the game}\label{sec 2}
Fix $\epsilon>0$ and $p>2$. The probabilities in the game are $\alpha=(p-2)/(n+p)$ and $\beta=(n+2)/(n+p)$. Define
\[
\Gamma_\epsilon :=\left\{x\in \R^n\setminus\Omega\ :\ \text{dist}(x,\partial \Omega)\leq \epsilon\right\}
\]
and
\[
\Omega_\epsilon:=\Omega\cup \Gamma_\epsilon.
\] 
Then $B_\epsilon(x)\subset \Omega_\epsilon$ for all $x\in \Omega$. The game ends when the token hits $\Gamma_\epsilon$ for the first time. In Sections \ref{sec 2}, \ref{sec 3} and \ref{sec 4} the payoffs $F:\Gamma_\epsilon\rightarrow \R$ and $f:\Omega\rightarrow (0,\infty)$ are bounded and Borel measurable.  

Let us briefly describe the stochastic terminology used in this paper. Strategies $S_\text{I}$ for Max and $S_{\text{II}}$ for Minnie are collections of Borel-measurable functions that give the next game position given the history of the game. When we fix a certain strategy for a player, we usually write $S^\text{Max}_\text{I}$ for Max and $S^\text{Min}_\text{II}$ for Minnie. By a history of the game up to step $k$ we mean a sequence    
\[
(x_0,(c_1,x_1),...,(c_k,x_k)),
\]
where $x_0,...,x_k\in \Omega_\epsilon$ are game positions and $c_j\in \mathcal{C}:=\left\{0,1,2\right\}$. Here $c_j=0$ means that Max wins, $c_j=1$ that Minnie wins and $c_j=2$ that a random step occurs. Our probability space is the space of all game sequences
\[
H^\infty:=\left\{\omega\ :\ \omega\in x_0\times (\mathcal{C},\Omega_\epsilon)\times ...\right\}.
\] 
Put $\mathcal{F}_0:=\sigma(x_0)$ and 
\[
\mathcal{F}_k:=\sigma(x_0,(c_1,x_1),...,(c_k,x_k))
\]
for $k\geq 1$. Note that here $(c_1,x_1),...$ are random variables. Then 
\[
\tau(\omega):=\inf \left\{k\ :\ x_k\in \Gamma_\epsilon, k=0,1,...\right\}
\]
is a stopping time relative to the filtration $\left\{\mathcal{F}_k\right\}^\infty_{k=0}$. 

The fixed starting point $x_0$ and the strategies $S_{\text{I}}$ and $S_{\text{II}}$ determine a unique probability measure $\mathbb{P}^{x_0}_{S_{\text{I}},S_{\text{II}}}$ on the product $\sigma$-algebra, see e.g.\ \cite{Harnack}.

The expected total payoff, when starting from $x_0$ and using strategies $S_\text{I}$ and $S_\text{II}$, is obtained as a sum of final payoff and running payoff
\begin{align*}
\mathbb{E}^{x_0}&_{S_\text{I},S_\text{II}}\left[F(x_\tau)+\epsilon^2\sum^{\tau-1}_{i=0}f(x_i)\right]\\ 
& :=\int_{H^\infty}\left(F(x_\tau (\omega))+\epsilon^2\sum^{\tau-1}_{i=0}f(x_i)\right)d\mathbb{P}^{x_0}_{S_\text{I},S_\text{II}}(\omega).
\end{align*}

\emph{The value of the game for Max in} $x_0\in \Omega$ is given by
\[
u^\epsilon_\text{I}(x_0):=\sup_{S_\text{I}}\inf_{S_\text{II}}\mathbb{E}^{x_0}_{S_\text{I},S_\text{II}}\left[F(x_\tau)+\epsilon^2\sum^{\tau-1}_{i=0}f(x_i)\right],
\]
while \emph{the value of the game for Minnie} is given by
\[
u^\epsilon_\text{II}(x_0):=\inf_{S_\text{II}}\sup_{S_\text{I}}\mathbb{E}^{x_0}_{S_\text{I},S_\text{II}}\left[F(x_\tau)+\epsilon^2\sum^{\tau-1}_{i=0}f(x_i)\right].
\]
If a function $u$ is defined in $\Omega_\epsilon$, $u=F$ on $\Gamma_\epsilon$ and
\[
u=u^\epsilon_\text{I}=u^\epsilon_\text{II}
\]
in $\Omega$, then $u$ is \emph{the value of the game}.

The next two lemmas guarantee that the game has a value which satisfies DPP \eqref{dynaamista}. For similar results without a running payoff, see \cite[Theorems 2.1, 2.2, 3.2]{ex}.

\begin{lem}\label{DPP1}
For given payoffs and $\epsilon>0$, there is a unique Borel measurable function $u_\epsilon:\Omega_\epsilon\rightarrow \R$, $u_\epsilon=F$ on $\Gamma_\epsilon$, which satisfies DPP \eqref{dynaamista} for all $x\in \Omega$.
\end{lem} 

\begin{proof}
First we show the existence. Let $(u_k)^\infty_{k=0}$ be a sequence of functions $\Omega_\epsilon\rightarrow \R$ such that $u_k=F$ on $\Gamma_\epsilon$ for all $k\in \N$, $u_0=\inf_{\Gamma_\epsilon}F$ in $\Omega$ and
\[
u_{k+1}(x)=\frac{\alpha}{2}\left\{\sup_{B_\epsilon(x)}u_k+\inf_{B_\epsilon(x)}u_k\right\}+\beta\vint_{B_\epsilon(x)}u_k dy+\epsilon^2 f(x)
\]
for all $k\in \N$ and $x\in \Omega$. If $k\geq 1$ and $u_k\geq u_{k-1}$, then 
\begin{align*}
u_{k+1}(x)&\geq \frac{\alpha}{2}\left\{\sup_{B_\epsilon(x)}u_{k-1}+\inf_{B_\epsilon(x)}u_{k-1}\right\}+\beta\vint_{B_\epsilon(x)}u_{k-1} dy+\epsilon^2 f(x)\\
&=u_k(x)
\end{align*}
for all $x\in \Omega$. Since $f>0$, we have $u_1\geq u_0$. By induction, the sequence $(u_k)$ is increasing. 

The sequence $(u_k)$ is also bounded. Let $D=\text{diam}(\Omega)$ and $N=\sup_\Omega f$. Note that for any point $y_0\in \Omega$ there is a sequence $(y_i)^{2D/\epsilon}_{i=0}$ for which $\left|y_{i+1}-y_i\right|\leq \epsilon/2$ and $y_{2D/\epsilon}\in \Gamma_\epsilon$. Choose arbitrary $k_0\in \N$. We may assume 
\[
\sup_\Omega u_{k_0}\geq \sup_{\Gamma_\epsilon}F. 
\]
Choose $x_0\in \Omega$ such that
\[
u_{k_0}(x_0)>\left(1-\frac12 \left(\frac\alpha2\right)^{2D/\epsilon}\right)\sup_\Omega u_{k_0}.
\] 
Let $(x_j)^{J}_{j=0}\subset \Omega_\epsilon$ be a sequence for which $\left|x_{j+1}-x_j\right|\leq\epsilon/2$, $x_J\in \Gamma_\epsilon$ and $J\leq 2D/\epsilon$. By using the rough estimates
\[
\sup_{\Omega_\epsilon}u_{k_0-1}\leq \sup_\Omega u_{k_0},
\] 

\[
\inf_{B_\epsilon(x_j)}u_{k_0-1}\leq u_{k_0}(x_j)
\]
and DPP we obtain
\[
u_{k_0}(x_0)\leq \left(\frac\alpha2+\beta\right)\sup_\Omega u_{k_0}+\frac\alpha2 u_{k_0}(x_1)+\epsilon^2 N.
\] 
Repeating this estimate for the values $u_{k_0}(x_j)$, $j\in\left\{1,...,J\right\}$, we get
\begin{align*}
u_{k_0}(x_0)&\leq \left(\frac\alpha2+\beta\right)\sup_\Omega u_{k_0}\sum^{2D/\epsilon}_{j=0}\left(\frac\alpha2\right)^j+\epsilon^2N\sum^{2D/\epsilon}_{j=0}\left(\frac\alpha2\right)^j\\
& \leq \left(1-\left(\frac\alpha2\right)^{2D/\epsilon}\right)\sup_\Omega u_{k_0}+2\epsilon^2N.
\end{align*}
Remembering how $x_0$ was chosen, we have
\[
\sup_\Omega u_{k_0}\leq 4\epsilon^2N\left(\frac2\alpha\right)^{2D/\epsilon}.
\]
Since $k_0$ was arbitrary and the right hand side does not depend on $k_0$, the sequence $(u_k)$ is bounded. Hence it converges pointwise to a bounded, Borel measurable limit function $u$. We show that the convergence is uniform. Suppose not. Since a sequence $\sup_{\Omega_\epsilon}(u-u_k)$ is positive, decreasing and bounded, we have
\[
M=\lim_{k}\sup_{\Omega_\epsilon}(u-u_k)>0.
\]
If $1\leq k_1\leq k_2$, we have
\[
\max\left\{\sup_{B_\epsilon (x)}u_{k_2}-\sup_{B_\epsilon (x)}u_{k_1},\inf_{B_\epsilon (x)}u_{k_2}-\inf_{B_\epsilon (x)}u_{k_1}\right\}\leq \sup_{B_\epsilon (x)}\left(u-u_{k_1}\right)
\]
for all $x\in \Omega$. Using DPP we estimate
\begin{equation}\label{delta}
u_{k_2+1}(x)-u_{k_1+1}(x)\leq \alpha \sup_{B_\epsilon(x_0)}(u-u_{k_1})+\beta \vint_{B_\epsilon(x)}(u-u_{k_1})dz. 
\end{equation}
Fix $\delta>0$. Select $k_1$ such that
\[
\sup_{\Omega_\epsilon}(u-u_{k_1})<M+\delta
\]
and
\[
\sup_{x\in \Omega}\beta \vint_{B_\epsilon(x)}(u-u_{k_1})dz\leq \delta.
\]
Then pick $y\in \Omega$ such that $u(y)-u_{k_1}(y)> M-\delta$, and finally $k_2\geq k_1$ such that $u(y)-u_{k_2}(y)<\delta$. Then
\[
u_{k_2+1}(y)-u_{k_1+1}(y)>M-2\delta,
\]
and since $\alpha<1$, the estimate \eqref{delta} contradicts the assumption $M>0$ when $\delta$ is small enough. Since the convergence is uniform, the limit function $u$ satisfies DPP \eqref{dynaamista}.

In the proof of uniqueness the running payoff plays a minor role, so we just explain the ideas and refer to the proof of \cite[Theorem 2.2]{ex} for details. Assume that $u$ and $v$ are defined in $\Omega_\epsilon$, satisfy DPP in $\Omega$ and $u=F=v$ on $\Gamma_\epsilon$. Assume that $u(y)>v(y)$ for some $y\in \Omega$. Since $u-v$ is bounded, we have
\[
\sup_{\Omega}(u-v)=:M>0.
\]    
Using DPP, we can estimate
\begin{align*}
u(x)-v(x)&\leq \frac\alpha2\left(\sup_{B_\epsilon (x)}u(x)-\sup_{B_\epsilon (x)}v(x)\right)+\frac{\alpha}{2}\left(\inf_{B_\epsilon (x)}u(x)-\inf_{B_\epsilon (x)}v(x)\right)\\
&\phantom{{}=\frac\alpha2}+\beta\vint_{B_\epsilon (x)}(u-v)dz+f(x)-f(x)\notag\\
& \leq \alpha M+\beta\vint_{B_\epsilon (x)}(u-v)dz.
\end{align*}
Because of absolute continuity of the integral, a set 
\[
G:=\left\{x: u(x)-v(x)=M\right\} 
\]
is non-empty, and if $x_0\in G$, then $u-v=M$ almost everywhere in a ball $B_\epsilon(x_0)$. This contradicts the assumption that $G$ is bounded. A similar contradiction follows if $v(y)>u(y)$ for some $y\in \Omega$. Hence $u=v$ in $\Omega_\epsilon$. 
\end{proof}

\begin{lem}\label{DPP2}
Given the payoffs and $\epsilon>0$, the tug-of-war with noise and running payoff has a unique value function $u_\epsilon:=u^\epsilon_\text{I}=u^\epsilon_\text{II}$ which satisfies DPP \eqref{dynaamista} in $\Omega$.
\end{lem}

\begin{proof}
By Lemma \ref{DPP1}, there is a unique function $u_\epsilon$, $u_\epsilon=F$ in $\Gamma_\epsilon$, satisfying DPP \eqref{dynaamista}. We show that
\[
u^\epsilon_\text{II}\leq u_\epsilon\leq u^\epsilon_\text{I}\leq u^\epsilon_\text{II}.
\]
Since 
\[
\sup_{S_\text{I}}\mathbb{E}^{x_0}_{S_\text{I},S_\text{II}}\left[F(x_\tau)+\epsilon^2\sum^{\tau-1}_{i=0}f(x_i)\right]\geq u^\epsilon_\text{I}
\] 
for all strategies $S_\text{II}$, we get $u^\epsilon_\text{I}\leq u^\epsilon_\text{II}$. 

Next we show that $u_\epsilon\leq u^\epsilon_\text{I}$. Max follows a strategy $S^{\text{Max}}_\text{I}$ in which from $x_{k-1}$ he steps to a point $x_k\in B_\epsilon (x_{k-1})$ so that for fixed $\eta>0$
\[
u_\epsilon(x_k)\geq \sup_{B_\epsilon (x_{k-1})} u_\epsilon-\eta 2^{-k}.
\]
Minnie uses a strategy $S_\text{II}$. Using DPP for $u_\epsilon$ at a point $x_{k-1}$, we estimate
\begin{align*}
\mathbb{E}^{x_0}_{S^\text{Max}_\text{I},S_\text{II}}&[u_\epsilon(x_k)+\epsilon^2\sum^{k-1}_{i=0}f(x_i)-\eta 2^{-k}|\mathcal{F}_{k-1}]\\
& \geq \frac{\alpha}{2}\left\{\sup_{B_\epsilon(x_{k-1})}u_\epsilon-\eta 2^{-k}+\inf_{B_\epsilon(x_{k-1})}u_\epsilon\right\}\\
&\phantom{{}=\frac{\alpha}{2}}+\beta \vint_{B_\epsilon(x_{k-1})}u_\epsilon dy+\epsilon^2f(x_{k-1})+\epsilon^2\sum^{k-2}_{i=0}f(x_i)-\eta 2^{-k}\notag\\
& =u_\epsilon(x_{k-1})+\epsilon^2\sum^{k-2}_{i=0}f(x_i)-\eta 2^{-k}(1+\alpha/2)\\
& \geq u_\epsilon(x_{k-1})+\epsilon^2\sum^{k-2}_{i=0}f(x_i)-\eta 2^{-(k-1)}.
\end{align*}
Therefore the process 
\[
M_k:=u_\epsilon(x_k)+\epsilon^2\sum^{k-1}_{i=0}f(x_i)-\eta 2^{-k} 
\]
for $k\geq 1$, $M_0=u_\epsilon(x_0)-\eta$, is a submartingale with respect to the strategies $S^\text{Max}_\text{I}$ and $S_\text{II}$. Using the Optional Stopping Theorem we obtain 
\begin{align*}
u^\epsilon_{\text{I}}(x_0)&=\sup_{S_\text{I}}\inf_{S_\text{II}}\mathbb{E}^{x_0}_{S_\text{I},S_\text{II}}\left[F(x_\tau)+\epsilon^2\sum^{\tau-1}_{i=0}f(x_i)\right]\\
& \geq \inf_{S_\text{II}}\mathbb{E}^{x_0}_{S^{\text{Max}}_\text{I},S_\text{II}}\left[F(x_\tau)+\epsilon^2\sum^{\tau-1}_{i=0}f(x_i)-\eta 2^{-\tau}\right]\\
& =\inf_{S_\text{II}}\mathbb{E}^{x_0}_{S^{\text{Max}}_\text{I},S_\text{II}}\left[u_\epsilon(x_\tau)+\epsilon^2\sum^{\tau-1}_{i=0}f(x_i)-\eta 2^{-\tau}\right]\\
& \geq \inf_{S_\text{II}}\mathbb{E}^{x_0}_{S^{\text{Max}}_\text{I},S_\text{II}}\left[u_\epsilon (x_0)-\eta\right]\\
& = u_\epsilon (x_0)-\eta.
\end{align*}
Since $\eta>0$ was arbitrary, we have $u^\epsilon_{\text{I}}(x_0)\geq u_\epsilon (x_0)$. Inequality 
\[
u^\epsilon_{\text{II}}(x_0)\leq u_\epsilon (x_0) 
\]
follows from symmetric argument, so
\[
u_\epsilon=u^\epsilon_\text{I}=u^\epsilon_\text{II}
\]
in $\Omega$. Hence $u_\epsilon$ is the value of the game.
\end{proof}

The next two lemmas are useful tools in estimating the value function.

\begin{lem}\label{2.4} Let $\tau$ be the stopping time of the game and let $\tau^*\leq \tau$ be a stopping time with respect to the filtration $\mathcal{F}_k$. Then
\[
u_\epsilon(y)\geq \inf_{S_\text{\emph{II}}}\mathbb{E}^y_{S^\text{\emph{Max}}_\text{\emph{I}},S_\text{\emph{II}}}\left[u_\epsilon(x_{\tau^*})+\epsilon^2\sum^{\tau^*-1}_{i=0}f(x_i)\right]
\] 
for any fixed strategy $S^\text{Max}_\text{\emph{I}}$, and
\[
u_\epsilon(y)\leq \sup_{S_\text{\emph{I}}}\mathbb{E}^y_{S_\text{\emph{I}},S^\text{\emph{Min}}_\text{\emph{II}}}\left[u_\epsilon(x_{\tau^*})+\epsilon^2\sum^{\tau^*-1}_{i=0}f(x_i)\right]
\]
for any fixed strategy $S^\text{Min}_\text{\emph{II}}$.
\end{lem} 
 
\begin{proof}
We only prove the first inequality, since the second follows from similar argument. Max has fixed a strategy $S^\text{Max}_\text{I}$. Let $\eta>0$. Minnie follows a strategy $S^\text{Min}_{\text{II}}$ in which from $x_{k-1}\in \Omega$ she steps to a point $x_k\in B_\epsilon(x_{k-1})$ in which
\[
u_\epsilon(x_k)\leq \inf_{B_\epsilon(x_{k-1})}u_\epsilon+\eta 2^{-k}.
\]
Let us first prove that
\[
M_k:=u_\epsilon(x_k)+\epsilon^2\sum^{k-1}_{i=0}f(x_i)+\eta 2^{-k}
\]
for $k\geq 1$, $M_0=u_\epsilon(x_0)+\eta$, is a supermartingale under the strategies $S^\text{Max}_\text{I}$ and $S^{\text{Min}}_\text{II}$. 
\begin{align*}
\mathbb{E}^{x_0}_{S^\text{Max}_\text{I},S^\text{Min}_\text{II}}&[M_k|\mathcal{F}_{k-1}]\\
& \leq \frac{\alpha}{2}\left\{\sup_{B_\epsilon (x_{k-1})}u_\epsilon+\inf_{B_\epsilon (x_{k-1})}u_\epsilon+\eta 2^{-k}\right\}\\
&\phantom{{}=\frac{\alpha}{2}}+\beta\vint_{B_\epsilon (x_{k-1})} u_\epsilon dy +\epsilon^2\sum^{k-1}_{i=0}f(x_i)+\eta 2^{-k}\notag\\
& \leq u_\epsilon (x_{k-1})+\epsilon^2\sum^{k-2}_{i=0}f(x_i)+\eta 2^{-(k-1)}=M_{k-1}.
\end{align*}
Hence $M_k$ is a supermartingale, and we get
\begin{align*} \inf_{S_\text{II}}\ &\mathbb{E}^y_{S^\text{Max}_\text{I},S_{\text{II}}}\left[u_\epsilon(x_{\tau^*})+\epsilon^2\sum^{\tau^*-1}_{i=0}f(x_i)\right]\\
& \leq \mathbb{E}^y_{S^\text{Max}_\text{I},S^\text{Min}_{\text{II}}}\left[u_\epsilon(x_{\tau^*})+\epsilon^2\sum^{\tau^*-1}_{i=0}f(x_i)+\eta 2^{-\tau^*}\right]\\
& \leq \mathbb{E}^y_{S^\text{Max}_\text{I},S^\text{Min}_{\text{II}}}[M_0]=u_\epsilon (y)+\eta.
\end{align*}
Since $\eta>0$ was arbitrary, the result follows.
\end{proof}

\begin{lem}
If $v_\epsilon$ and $u_\epsilon$ are value functions with payoff functions $f_v$ and $F_v$ for $v_\epsilon$, $f_u$ and $F_u$ for $u_\epsilon$, and $f_v\geq f_u$, $F_v\geq F_u$, then $v_\epsilon\geq u_\epsilon$.
\end{lem}

\begin{proof}
Max plays with a strategy $S_\text{I}$ and Minnie follows a strategy $S^\text{Min}_{\text{II}}$ in which from $x_{k-1}\in \Omega$ she steps to a point $x_k\in B_\epsilon(x_{k-1})$ in which
\[
v(x_k)\leq \inf_{B_\epsilon(x_{k-1})}v+\eta2^{-k}
\]
for some fixed $\eta>0$. Then
\begin{align*}
\mathbb{E}^{x_0}_{S_\text{I},S^\text{Min}_{\text{II}}}&\left[v(x_k)+\epsilon^2\sum^{k-1}_{i=0}f_u(x_i)+\eta2^{-k}\ |\mathcal{F}_{k-1}\right]\\
& \leq \frac{\alpha}{2}\left\{\inf_{B_\epsilon(x_{k-1})}v+\eta2^{-k}+\sup_{B_\epsilon(x_{k-1})}v\right\}+\beta\vint_{B_\epsilon(x_{k-1})}vdy\\
&\phantom{{}=\frac{\alpha}{2}}+\epsilon^2\sum^{k-1}_{i=0}f_u(x_i)+\eta2^{-k}\notag\\
& \leq v(x_{k-1})+\epsilon^2\sum^{k-1}_{i=0}f_u(x_i)-\epsilon^2f_v (x_{k-1})+\eta2^{-(k-1)}\\
& \leq v(x_{k-1})+\epsilon^2\sum^{k-2}_{i=0}f_u(x_i)+\eta2^{-(k-1)},
\end{align*}
since $v$ is a value function and $f_v\geq f_u$. Thus
\[
M_k=v(x_k)+\epsilon^2\sum^{k-1}_{i=0}f_u(x_i)+\eta2^{-k}
\]
for $k\geq 1$, $M_0=v_\epsilon(x_0)+\eta$, is a supermartingale. Since $F_v\geq F_u$ on $\Gamma_\epsilon$, we deduce by the Optional Stopping Theorem that
\begin{align*}
u_\epsilon (x_0)&= \inf_{S_{\text{II}}}\sup_{S_{\text{I}}}\mathbb{E}^{x_0}_{S_{\text{I}},S_{\text{II}}}\left[F_u(x_\tau)+\epsilon^2\sum^{\tau-1}_{i=0}f_u (x_i)\right]\\
& \leq \sup_{S_{\text{I}}}\mathbb{E}^{x_0}_{S_{\text{I}},S^\text{Min}_{\text{II}}}\left[F_v (x_\tau)+\epsilon^2\sum^{\tau-1}_{i=0}f_u (x_i)+\eta2^{-\tau}\right]\\
& \leq \sup_{S_{\text{I}}}\mathbb{E}^{x_0}_{S_{\text{I}},S^\text{Min}_{\text{II}}}[M_0]=v(x_0)+\eta.
\end{align*}
Since $\eta$ was arbitrary, this proves the claim.
\end{proof}

\section{Stopping time estimates and regularity lemmas}\label{sec 3}
Recall that since the running payoff is positive, the value function $u_\epsilon$ is bounded from below by $\inf_{\Gamma_\epsilon}F$. In the proof of Lemma \ref{DPP1} we saw that $u_\epsilon$ is bounded from above by
\[
\max\left\{\sup_{\Gamma_\epsilon}F,\ 4\epsilon^2\left(\frac2\alpha\right)^{2\diam(\Omega)/\epsilon}\sup_\Omega f\right\}.
\]
Unfortunately, this upper bound depends on $\epsilon$. Using the lemmas of section \ref{sec 2}, we can now show that the value functions $u_\epsilon$ for different $\epsilon$ are uniformly bounded.  The idea is to fix for Minnie a strategy in which she tries to push the token to a certain boundary point. No matter which strategy Max uses, the expected value of the stopping time can be estimated so that the total effect of the running payoff is under control.

\begin{lem}\label{maks} For given payoffs $F$ and $f$, there is $C>0$, independent of $\epsilon$, such that
\[
u_\epsilon\leq C(\sup_{\Gamma_\epsilon}F+\sup_{\Omega}f).
\]
\end{lem}

\begin{proof}
Fix $\epsilon>0$ and let $x_0\in \Omega$. Choose $z\in \R^n\setminus \Omega_\epsilon$, then $r>0$ such that $B_r(z)\subset \R^n\setminus\Omega_\epsilon$, and finally $R>0$ such that $\Omega_\epsilon\subset B_{R/2}(z)$. Let $v$ be a solution to the problem
\begin{displaymath}
\left\{ \begin{array}{ll}
\Delta v=-2(n+2) & \textrm{in}\ B_{R+\epsilon}\setminus \overline{B}_{r}(z),\\
v=0 & \text{on}\ \partial B_{r}(z),\\
\frac{\partial v}{\partial \nu}=0 & \text{on}\ \partial B_{R+\epsilon}(z), 
\end{array} \right.
\end{displaymath}
where $\frac{\partial v}{\partial \nu}$ is the normal derivative. As discussed in the proof of \cite[Lemma 4.5]{prop}, the function $v$ is concave in $r=|x-z|$, satisfies
\begin{equation}\label{poissoni}
v(x)=\vint_{B_\epsilon(x)}v\ dy+\epsilon^2
\end{equation}
and can be extended as a solution to the same equation to $\overline B_{r(z)}\setminus \overline B_{{r-\epsilon}(z)}$ so that equation \eqref{poissoni} holds also near the boundary $\partial B_r(z)$. 

%by putting $t=\left|x-z\right|$, a function $v(t)=-at^2-b\log(t)+c$ in dimension 2, $v(t)=-at^2-bt^{2-n}+c$ in higher dimensions

The game starts from $x_0\in \Omega$. Max plays with any strategy and Minnie plays with the strategy $S^\text{Min}_{\text{II}}$ in which from a point $x_{k-1}\in B_R(z)$ she moves to a point $x_k$ for which
\[
v(x_k)\leq \inf_{B_\epsilon(x_{k-1})}v+\frac\beta\alpha\epsilon^2.
\] 
Let $\tau$ be the smallest $k$ for which $x_k\in B_R(z)\setminus \Omega$ and $\tau^*$ the smallest $k$ for which $x_k$ hits the complement of $B_R(z)\setminus B_r(z)$. Then $\tau\leq \tau^*$ for any game sequence $(x_k)$. Let us estimate the expected value of $\tau^*$. By radial concavity of $v$ we get
\begin{align*}
\mathbb{E}^{x_0}_{S_{\text{I}},S^\text{Min}_{\text{II}}}\left[v(x_k)|\mathcal{F}_{k-1}\right]&\leq \frac{\alpha}{2}\left\{\sup_{B_\epsilon(x_{k-1})}v+\inf_{B_\epsilon(x_{k-1})}v+\frac\beta\alpha\epsilon^2\right\}\\
&\phantom{{}=\frac{\alpha}{2}}+\beta\vint_{B_\epsilon(x_{k-1})}v\ dy\notag\\
& \leq \alpha v(x_{k-1})+\frac{\beta}{2}\epsilon^2+\beta(v(x_{k-1})-\epsilon^2)\\
& = v(x_{k-1})-\frac{\beta}{2}\epsilon^2.
\end{align*}
Hence $M_k:=v(x_k)+k\frac{\beta}{2}\epsilon^2$ is a supermartingale. In particular, we have
\[
\mathbb{E}^{x_0}_{S_{\text{I}},S^\text{Min}_{\text{II}}} \left[M_{\tau^*}\right]\leq v(x_0)\leq C,
\]
where $C$ is independent of $\epsilon$. On the other hand, since $v(x_{\tau^*})=0$, we have $\mathbb{E}\left[ M_{\tau^*}\right]\geq \frac{\beta}{2}\epsilon^2\mathbb{E}\tau^*$. Hence 
\[
\mathbb{E}\left[\tau\right]\leq \mathbb{E}\left[\tau^*\right]\leq C\epsilon^{-2}.
\]
Then
\[
u_\epsilon(x_0)\leq \sup_{S_{\text{I}}}\mathbb{E}^{x_0}_{S_{\text{I}},S^\text{Min}_{\text{II}}}[F(x_\tau)+\epsilon^2\sum^{\tau-1}_{i=0}f(x_i)]\leq C(\sup_{\Gamma_\epsilon} F+\sup_{\Omega} f).
\]
Since $x_0\in \Omega$ and $\epsilon>0$ were arbitrary, the proof is complete.
\end{proof}

In the proof of Theorem \ref{Lemma 1} we need the following lemma, which is proven in the appendix of \cite{Harnack}.

Put a token to the point $(0,t)\in B_{2r}(0)\times [0,2r]\subset \R^{n+1}$ and fix $r>0$. From a point $(x_j,t_j)$, with probability $\alpha/2$ the token moves to the point $(x_j,t_j-\epsilon)$, and with the same probability to $(x_j,t_j+\epsilon)$. With probability $\beta$ the token moves to the point $(x_{j+1},t_j)$, where $x_{j+1}$ is randomly chosen from the ball $B_\epsilon(x_j)\subset \R^n$. 

\begin{lem}\label{sylinteri}
The probability that the token does not escape the cylinder through its bottom is less than
\[
C(p,n)(t+\epsilon)/r
\]
for all $\epsilon>0$ small enough.
\end{lem}

Also the next lemma is needed in the proof of Theorem \ref{Lemma 1}, because it describes the expected total effect of the running payoff under the strategies used there.

Let $0<\epsilon<t_0<1$ and start a random walk from $t_0$ as follows. From the point $t_{j-1}$ we step with probability $\frac{\alpha}{2}$ to $t_j=t_{j-1}+\epsilon$, with the same probability to $t_j=t_{j-1}-\epsilon$, and with probability $\beta$ we do not move, $t_j=t_{j-1}$. The random walk stops when $x_j\in \R\setminus (0,1)$ for the first time. Let $\overline{\tau}$ be the stopping time.
\begin{lem}\label{randomi}
In the random walk described above,
\[
\mathbb{E}[\overline{\tau}]\leq 5t_0 \alpha^{-1}\epsilon^{-2}.
\] 
\end{lem}

\begin{proof}
Since
\begin{align*}
\mathbb{E}[t^2_j|t_0,...,t_{j-1}]&=\frac{\alpha}{2}(t_{j-1}+\epsilon)^2+\frac{\alpha}{2}(t_{j-1}-\epsilon)^2+\beta t_{j-1}^2\\
& =t_{j-1}^2+\alpha \epsilon^2,
\end{align*} 
we have 
\[
\mathbb{E}[(t^2_j-j\alpha\epsilon^2)|t_0,...,t_{j-1}]=t^2_{j-1}-(j-1)\alpha\epsilon^2.
\]
Hence also $(t^2_j-j\alpha\epsilon^2)$ is a martingale. Let $p=\mathbb{P}(x_{\overline{\tau}}\leq 0)$. Then
\[
t_0=\mathbb{E}t_\tau\geq p(-\epsilon)+(1-p)=-p(\epsilon+1)+1.
\]
For the function $f:[0,\infty)\rightarrow \R$, $f(x)=(1-t_0)(1+x)^{-1}+x+t_0-1$ it holds that $f(0)=0$ and $f^{'}\geq 0$, so we have
\[
p\geq (1-t_0)(1+\epsilon)^{-1}\geq 1-t_0-\epsilon.
\]
Since $(t^2_j-j\alpha\epsilon^2)$ is a martingale, we have 
\begin{align*}
t^2_0&=t^2_0-0\cdot\alpha\epsilon^2=\mathbb{E}(t_{\overline{\tau}}^2-\overline{\tau}\alpha\epsilon^2)\\
& \leq p\epsilon^2+(1-p)(1+\epsilon)^2-\alpha\epsilon^2\mathbb{E}\overline{\tau}.
\end{align*}
We can estimate
\begin{align*}
\mathbb{E}\overline{\tau}&\leq \epsilon^{-2}(p\epsilon^2+(1-p)(1+\epsilon)^2-t^2_0)\\
& =\epsilon^{-2}\alpha^{-1}((1+\epsilon)^2-t^2_0+p\epsilon^2-p(1+\epsilon)^2)\\
& \leq \epsilon^{-2}\alpha^{-1}((1+\epsilon)^2-t^2_0-p)\\
& \leq \epsilon^{-2}\alpha^{-1}((1+\epsilon)^2-t^2_0+t_0+\epsilon-1)\\
& =\epsilon^{-2}\alpha^{-1}(\epsilon^2+3\epsilon+t_0-t^2_0)\\
& \leq \epsilon^{-2}\alpha^{-1}(t_0+4\epsilon)\\
& \leq 5t_0 \alpha^{-1}\epsilon^{-2}. 
\end{align*}
\end{proof}

The next two lemmas are needed in the proof of Harnack's inequality. The first is a simple local comparison estimate, and the second gives estimates for $\inf u_\epsilon$ in balls of radius $2\epsilon<r<1$. 

\begin{lem}\label{4.1}
Let $u_\epsilon>0$ be a value function and $x,y\in B_R(z)\subset \Omega$, $\left|x-y\right|\leq 10\epsilon$. Then
\[
u_\epsilon (x)\geq (\alpha/2)^{20}u_\epsilon (y).
\]
\end{lem}

\begin{proof}
Start the game from $x$. Max uses a strategy $S^\text{Max}_{\text{I}}$ in which he takes $\frac{\epsilon}{2}$-step towards $y$, and jumps to $y$ if possible. The game is stopped when the token reaches either $y$ or $\Omega_\epsilon\setminus B_R(z)$. Let this stopping time be $\tau^*$. Since the probability to stop at $y$ is bigger than $(\alpha/2)^{20}$, we obtain from Lemma \ref{2.4} 
\[
u_\epsilon(x)\geq \inf_{S_{\text{II}}}\mathbb{E}^{x_0}_{S^\text{Max}_{\text{I}},S_{\text{II}}}[u_\epsilon(x_{\tau^*})]+20\epsilon^2 \inf_\Omega f\geq (\alpha/2)^{20}u_\epsilon (y). 
\]
\end{proof}

\begin{lem}\label{4.2}
Let $u_\epsilon>0$ be a value function and $B_{30R}(y)\subset \Omega$ for some $R>0$. For $z\in B_{2R}(y)$ and $r\in (2\epsilon,R)$
\[
\inf_{B_r(z)}u_\epsilon\leq Cr^{-n}u_\epsilon(y),
\]
where $C=C(p,n)$
\end{lem}

\begin{proof}
Without a loss of generality, we may assume that $y=0$ and $R=1$. Fix $\epsilon>0$ and $r\in (2\epsilon,1)$. Let $U=B_4(z)\setminus\overline{B}_r(z)$. There is no loss of generality in assuming that $0\in U$. 

Define 
\begin{displaymath}
v(x) = \left\{ \begin{array}{ll}
(\left|x-z\right|^{2-n}-4^{2-n})(r^{2-n}-4^{2-n})^{-1} & \textrm{if $n\geq 3$,}\\
\log(4/\left|x-z\right|)\log (4/r)^{-1} & \textrm{if $n=2$.}
\end{array} \right.
\end{displaymath}
Then $v$ is harmonic in $U$ with boundary values 
\begin{equation}\nonumber
\begin{cases}
v=1\ \text{on}\ \partial B_r(z), \\
v=0\ \text{on}\ \partial B_4(z).
\end{cases}
\end{equation}
In both the cases there is a constant $c>0$ such that
\[
v(0)\geq cr^n.
\]

The game starts from $x_0=0$. Minnie uses any strategy and Max uses the following strategy $S^\text{Max}_I$: In a ball $B_\epsilon (x_{k-1})$, he aims to a point $x_k$ where   
\[
v(x_k)\geq \sup_{B_\epsilon (x_{k-1})}v-\eta r^n \epsilon^2,
\]
where $\eta>0$ is selected so that the stopping time estimation of the proof of Lemma \ref{maks} is at our disposal. The game is stopped at the $\epsilon$-boundary $\Gamma_\epsilon$ of $U$ and employing the boundary values $(u_{\epsilon})_{|\Gamma_\epsilon}$. The corresponding stopping time is $\tau^*$. 

We want to estimate the probability of stopping at the inner boundary. 

Since $y\mapsto v(x+y)$ is radially decreasing and the map $0<t\mapsto v(x+ty_0)$ is convex for any fixed $y_0\neq 0$, we obtain
\begin{align*}
& \mathbb{E}^{x_0}_{S^\text{Max}_{\text{I}},S_{\text{II}}}[v(x_{k+1})|x_0,x_1,...,x_k]\\
& \geq \frac{\alpha}{2}\left\{\sup_{B_\epsilon(x_k)}v-\eta r^n\epsilon^2+\inf_{B_\epsilon(x_k)}v\right\}+\beta \vint_{B_\epsilon (x_k)}vdy\\
& \geq \alpha v(x_k)+\beta v(x_k)-\eta r^n \epsilon^2=v(x_k)-\eta r^n \epsilon^2. 
\end{align*} 
Hence $M_k:=v(x_k)+k \eta r^n \epsilon^2$ is a submartingale. Denote by $P$ the probability of stopping at the inner boundary. The optional stopping theorem gives 
\[
cr^n\leq v(0)=v(x_0)\leq \mathbb{E}^{x_0}_{S^\text{Max}_\text{I},S_\text{II}}[v(x_{\tau^*})+\eta r^n \epsilon^2 \tau^*]\leq 2^{n+1}P+\eta C_1 r^n,
\]
where $C_1>0$ is a constant such that $\mathbb{E}\tau^*\leq C_1 \epsilon^{-2}$, and the term $2^{n+1}P$ comes from the fact that $v\leq 2^{n+1}$ in $B_r(z)\setminus B_{r-\epsilon}(z)$. We can select $\eta$ so that $\eta C_1<c$. Thus $P\geq c' r^n$, where $c'>0$. Using Lemma \ref{2.4} we obtain
\begin{align*}
u_\epsilon (0)=u_\epsilon (x_0)&\geq \inf_{S_\text{II}}\mathbb{E}^{x_0}_{S^\text{Max}_\text{I},S_\text{II}}[u_\epsilon (x_{\tau^*})+\epsilon^2 \tau^* \inf_\Omega f] \\
& \geq P \inf_{B_r(z)}u_\epsilon\\ 
& \geq c' r^n \inf_{B_r(z)}u_\epsilon.
\end{align*}
We get
\[
\inf_{B_r(z)}u_\epsilon\leq (c' r^n)^{-1}u_\epsilon(0)\leq Cr^{-n}u_\epsilon (0),
\]
where $C=c'^{-1}$. 
\end{proof}

\section{Lipschitz and Harnack estimates}\label{sec 4}
We are ready to prove the main results, Lipschitz continuity and Harnack's inequality. In the proof of the following theorem we use the cancellation strategy idea that was introduced in the proof of \cite[Theorem 3.2]{Harnack}. Because of Lemma \ref{randomi}, the running payoff behaves well under this strategy.
\begin{lause}\label{Lemma 1}
Let $u_\epsilon>0$ be a value function and $B_{6R}(a)\subset \Omega$, where $R>\epsilon$. When $\epsilon< r\leq R$, we have
\begin{equation}\label{oskillaatio}
\text{\emph{osc}}(u_\epsilon, B_r(a))\leq C\frac{r}{R}\left[\text{\emph{osc}}(u_\epsilon,B_{6R}(a))+\text{\emph{osc}}(f,B_{6r}(a))\right],
\end{equation}
where $C$ is a constant depending only on $p$ and $n$.
\end{lause}

\begin{proof}
Take $x,y\in B_r(a)$, $\left|x-y\right|\geq \epsilon$, and then $z\in B_{2r}(a)$ such that 
\[
\left|x-z\right|=\left|y-z\right|=\left|x-y\right|.
\]
When the game starts from $x$, Minnie plays according to the following cancellation strategy $S^\text{Min}_{\text{II}}$: If Max has won more coin tosses than Minnie, then she cancels one of the moves of Max. Otherwise, she moves towards $z$ length $\epsilon/2$ or keeps the token in $z$. We stop the game if Minnie wins $3\left|x-z\right|/\epsilon$ coin tosses more than Max, or if Max wins at least $2R/\epsilon$ times more than Minnie, or if the length of the sum of random vectors exceeds $2R$. Then the game stays in $B_{6R}$. For the game that starts from $y$, Max follows the cancellation strategy $S^\text{Max}_{\text{I}}$, and we define stopping time $\tau^*$ as previously. For this stopping time $\tau^*\leq \tau$, where $\tau$ is the normal stopping time of the game. Hence Lemma \ref{2.4} is at our disposal.

Notice that by putting 
\[
t_0=\frac{3\left|x-y\right|}{3\left|x-y\right|+2R},
\] 
Lemma \ref{randomi} gives
\[
\mathbb{E}[\tau^*]\leq 5\alpha^{-1}\frac{3\left|x-y\right|}{3\left|x-y\right|+2R}\epsilon^{-2}\leq C_1 \frac{\left|x-y\right|}{R}\epsilon^{-2}. 
\] 

Let $P$ be the probability that the game, started from $x$, ended because Minnie won more (hereafter Min w). By symmetry, $P$ is also the probability that the game, started from $y$, ended because Max won more (hereafter Max w). Then, because of the cancellation effect, by using Lemma \ref{2.4} we can estimate
\begin{align*}
|u_\epsilon(x)&-u_\epsilon (y)|\\
&\leq \left|P(\mathbb{E}^x_{S_\text{I},S^\text{Min}_\text{II}}\left[u_\epsilon(x_{\tau})|\text{Min w}\right]-\mathbb{E}^y_{S^\text{Max}_\text{I},S_\text{II}}\left[u_\epsilon(x_{\tau})|\text{Max w}\right])\right|\\
&\phantom{{}=P} +(1-P)\text{osc}_{B_{6R}(z_0)}u_\epsilon+\epsilon^2\mathbb{E}[\tau^*]\text{osc}_{B_{6R}(z_0)}f\notag\\
& \leq (1-P)\text{osc}_{B_{6R}(z_0)}u_\epsilon+C_1\left|x-y\right|R^{-1}\text{osc}_{B_{6R}(z_0)}f\\
& \leq C\frac{r}{R}\left[\text{osc}(u_\epsilon,B_{6R}(a))+\text{osc}(f,B_{6R}(a))\right]
\end{align*}
for $C=C(p,n)$, since by Lemma \ref{sylinteri} we have
\[
1-P\leq C_2\left|x-y\right|/R,
\]
where $C_2$ depends only on $p$ and $n$.

If $x,y\in B_r$ and $\left|x-y\right|<\epsilon$, we can take a point $z\in B_r$ such that $\left|x-z\right|\geq \epsilon$ and $\left|z-y\right|\geq \epsilon$. By triangle inequality the estimate follows from previous estimate.
\end{proof}

Next we prove Harnack's inequality. The idea is to show that if Harnack's inequality does not hold for a fixed, large constant, then by iteration argument the value functions are unbounded when $\epsilon$ is small. The cumulative effect of oscillations of the running payoff during iteration seems to cause trouble, but surprisingly, it is not even necessary to require the running payoff to be continuous.   

\begin{lause}\label{päälause}
Let $u_\epsilon>0$ be a value function. Assuming $B_{30r}(a)\subset \Omega$, where $r>0$, there exists a positive constant $K$, depending only on $p$ and $n$, for which
\[
\sup_{B_r(a)} u_\epsilon\leq K(\inf_{B_r(a)} u_\epsilon+\sup_{\Omega}f).
\]
\end{lause}

\begin{proof}
Without a loss in generality, we may assume that $r=1$ and $a=0$. For convenience of notation, let
\[
N:=\sup_\Omega f.
\]

First we show that 
\begin{equation}\label{inffi}
\inf_{\overline{B}_1(0)} u_\epsilon>0.
\end{equation}
Suppose not. Then there is a converging sequence $(x_j)\subset \overline{B}_1(0)$, $x_j\rightarrow x_0$, such that $u_\epsilon(x_j)<1/j$. According to Lemma \ref{4.1}, 
\[
u_\epsilon(y)\geq (\alpha/2)^{20}u_\epsilon(x_0)
\]
when $\left|y-x_0\right|<10\epsilon$. This is a contradiction, so \eqref{inffi} holds. 

Pick first a point $x_1\in B_1(0)$ such that
\[
u_\epsilon(x_1)<2\inf_{B_1(0)} u_\epsilon,
\] 
and then a point $x_2\in B_{2}(x_1)$ such that
\[
M_1:=u_\epsilon (x_2)\geq \sup_{B_2 (x_1)}u_\epsilon-N.
\]
For $k\geq 2$, let $R_k=2^{1-k}$ and pick $x_{k+1}\in B_{R_k}(x_k)$ such that
\[
M_k:=u_\epsilon (x_{k+1})\geq \sup_{B_{R_k}(x_k)}u_\epsilon-N.
\] 
We are going to show that 
\begin{equation}\label{ei totta}
M_1< (2^{1+2n}C)^{1+2n}u_\epsilon(x_1)+2N, 
\end{equation}
where $C=C(p,n)$ is a constant such that Lemma \ref{4.2} and Theorem \ref{Lemma 1} are valid. 

On the contrary, suppose that inequality \eqref{ei totta} does not hold. Put $\delta :=(2^{1+2n}C)^{-1}$. Let us show by induction that the counter assumption yields
\begin{equation}\label{M}
M_k\geq 2C(\delta R_{k+1})^{-2n}u_\epsilon(x_1).
\end{equation}
Notice first that from a straightforward calculation we get
\begin{equation}\label{samat}
2C(\delta R_{k+1})^{-2n}=(2C\delta)^{-k+1}(2^{1+2n}C)^{1+2n}.
\end{equation}
By observation \eqref{samat} the case $k=1$ holds. Assume that \eqref{M} holds for $k\leq j$. Let $k\in \left\{2,...,j+1\right\}$. Then
\begin{equation}\label{inf puoli}
\inf_{B_{\delta R_k}(x_k)}u_\epsilon\leq C(\delta R_k)^{-2n}u_\epsilon (x_1)\leq \frac{M_{k-1}}{2}=\frac{u_\epsilon (x_k)}{2}, 
\end{equation}
where we used a weakened form of Lemma \ref{4.2} and the induction assumption that \eqref{M} holds for $k\leq j$.

By Theorem \ref{Lemma 1} 
\[
\text{osc}(u_\epsilon,B_{\delta R_k}(x_k))\leq C\delta(\text{osc}(u_\epsilon,B_{R_k}(x_k))+\text{osc}(f,B_{R_k}(x_k))),
\]
or in other words,
\begin{equation}\label{kojootti}
\text{osc}(u_\epsilon,B_{R_k}(x_k))\geq (C\delta)^{-1}\text{osc}(u_\epsilon,B_{\delta R_k}(x_k))-\text{osc}(f,B_{R_k}(x_k)).
\end{equation}
By using first \eqref{kojootti} and then \eqref{inf puoli} we obtain
\begin{align*}
M_k&\geq \text{osc}(u_\epsilon, B_{R_k}(x_k))-N\\
& \geq (C\delta)^{-1}(\sup_{B_{\delta R_k}(x_k)}u_\epsilon-\inf_{B_{\delta R_k}(x_k)}u_\epsilon)-\text{osc}(f,B_{R_k}(x_k))-N\\
& \geq (C\delta)^{-1}(u(x_k)-2^{-1}u(x_k))-2N\\
& = (2C\delta)^{-1}M_{k-1}-2N.
\end{align*}
Now we come to an important point, when we want an estimation between $M_{j+1}$ and $M_1$. At first glance it seems that the cumulative effect of the oscillation of running payoff is an issue, but it turns out to be under control. We get
\begin{align*}
M_{j+1}&\geq (2C\delta)^{-1}M_{j}-2N\\
& \geq (2C\delta)^{-1}\left[(2C\delta)^{-1}M_{j-1}-\text{osc}\left(f,B_{R_{j}}(x_{j})\right)\right]-2N\\
& \geq (2C\delta)^{-j}M_1-2N\sum^{j}_{i=1}(2C\delta)^{-i+1}.
\end{align*}
Remembering the counter assumption 
\[
M_1\geq (2^{1+2n}C)^{1+2n}u_\epsilon(x_1)+2N
\]
and noticing that $2C\delta=2^{-2n}<1/2$, we obtain 
\begin{align*}
M_{j+1}&\geq (2C\delta)^{-j}(2^{1+2n}C)^{1+2n}u_\epsilon(x_1)+2N\left[(2C\delta)^{-j}-\sum^{j}_{i=1}(2C\delta)^{-i+1}\right]\\
&\geq (2C\delta)^{-j}(2^{1+2n}C)^{1+2n}u_\epsilon (x_1).
\end{align*} 
Taking into account the observation \eqref{samat}, the induction is complete.

Take $k_0$ such that $\delta R_{k_0}\in (2\epsilon,4\epsilon]$. By Lemma \ref{4.1}, 
\[
\inf_{B_{\delta R_{k_0}}(x_{k_0})}u_\epsilon\geq (\alpha/2)^{20}\sup_{B_{\delta R_{k_0}}(x_{k_0})}u_\epsilon.
\]
By using Lemma \ref{4.2} and inequality \eqref{M} we obtain
\begin{align*}
(\alpha/2)^{-20}&\geq \frac{\sup_{B_{\delta R_{k_0}}(x_{k_0})}u_\epsilon}{\inf_{B_{\delta R_{k_0}}(x_{k_0})}u_\epsilon}\geq \frac{u_\epsilon (x_{k_0})}{C(\delta R_{k_0})^{-n}u_\epsilon (x_1)}=\frac{M_{k_0-1}}{C(\delta R_{k_0})^{-n}u_\epsilon (x_1)}\\
& \geq \frac{(2C\delta)^{2-k_0}(2^{1+2n}C)^{1+2n}}{C(\delta 2^{1-k_0})^{-n}}\\
& \geq \widehat{C}(C\delta 2^{n+1})^{-k_0}=\widehat{C}2^{nk_0},
\end{align*}
where $\widehat{C}$ is independent of $k_0$. This is a contradiction when $k_0$ is big enough, or in other words, when $\epsilon$ is small enough. Therefore inequality \eqref{ei totta} holds and we get 
\begin{align*}
\sup_{B_1(0)}u_\epsilon &\leq \sup_{B_2(x_1)}u_\epsilon\leq M_1+N\leq (2^{1+2n}C)^{1+2n}u_\epsilon(x_1)+3N\\
& \leq2(2^{1+2n}C)^{1+2n}\inf_{B_1(0)}u_\epsilon+3\sup_{\Omega}f\\
& \leq K(\inf_{B_1(0)}u_\epsilon+\sup_{\Omega}f),
\end{align*}
where $K$ depends only on $p$ and $n$.
\end{proof}

\section{Relation to PDE}\label{sec 5}
In this section we study local reqularity of viscosity solutions to the inhomogeneous $p$-Laplace equation
\begin{equation}\label{vissi}
-\Delta^N_p u=f
\end{equation}
in $\Omega$. As before, $p>2$. In the whole section $f>0$ is continuous and bounded in $\overline{\Omega}$, and boundary values of viscosity solutions are required to be continuous and bounded. Recall that 
\[
\Delta^N_p u=\frac{1}{p}\left|\nabla u\right|^{2-p}\Delta_p u
\]
is the normalized $p$-Laplacian. Here
\[
\Delta_p u=\text{div}(\left|\nabla u\right|^{p-2}\nabla u)=\left|\nabla u\right|^{p-2}((p-2)\Delta^N_\infty u+\Delta u),
\]
where
\[
\Delta^N_\infty u=\left|\nabla u\right|^{-2}\Delta_\infty u=\left|\nabla u\right|^{-2}\left\langle D^2 u\ \nabla u,\nabla u\right\rangle.
\]

By \cite[Proposition 3]{Kawohl}, we can define viscosity solutions to \eqref{vissi} as follows.   

\begin{maar}\label{mar} 
A continuous function $u$ is a viscosity solution to \eqref{vissi} at $x\in \Omega$, if and only if every $C^2$-function $\phi$, $\nabla \phi(x)\neq 0$ or $D^2 \phi(x)=0$, that touches $u$ from below in $x\in \Omega$, satisfies 
\[
-\Delta^N_p \phi(x)\geq f(x),
\]
and every $C^2$-function $\phi$, $\nabla \phi(x)\neq 0$ or $D^2 \phi(x)=0$, that touches $u$ from above in $x\in \Omega$, satisfies 
\[
-\Delta^N_p \phi(x)\leq f(x).
\]
\end{maar}

Note that if a test function $\phi$ satisfies $\nabla \phi (x)=0$ and $D^2 \phi(x)=0$ for some $x\in \Omega$, by the convergence argument explained in \cite{Kawohl} we can set 
\[
\Delta^N_p \phi(x)=0.
\]  

The idea for showing local regularity properties for viscosity solutions to \eqref{vissi} is to notice that viscosity solutions can be approximated uniformly by value functions of tug-of-war with noise and running payoff. We need the following Arzela-Ascoli-type lemma, which is proven in \cite[Lemma 4.2]{prop}.

\begin{lem}\label{Ascoli}
Let $\left\{u_\epsilon:\overline{\Omega}\rightarrow \R,\ \epsilon>0\right\}$ be a uniformly bounded set of functions such that given $\eta>0$, there are constants $r_0$ and $\epsilon_0$ such that for every $\epsilon<\epsilon_0$ and any $x,y\in \overline{\Omega}$ with $\left|x-y\right|<r_0$ it holds that 
\[
\left|u_\epsilon(x)-u_\epsilon(y)\right|<\eta. 
\]
Then there exists a uniformly continuous function $u:\overline{\Omega}\rightarrow \R$ and a subsequence still denoted by ${u_\epsilon}$ such that $u_\epsilon\rightarrow u$ uniformly in $\overline{\Omega}$ as $\epsilon\rightarrow 0$.
\end{lem}

Let $u$ be a viscosity solution to \eqref{vissi} in $\Omega$. We may assume $0\in \Omega$. Choose $R>0$ such that $B_{2R}(0)\subset \Omega$. Let $u_\epsilon$, $0<\epsilon<R$, be the value function of tug-of-war with noise and running payoff in $B_R(0)$, where running payoff is $f$ from equation \eqref{vissi}, and final payoff is $u$ on the boundary strip
\[
\Gamma_\epsilon=\left\{x\in \Omega\setminus B_R(0)\ :\ \text{dist}(x,\partial B_R(0))\leq \epsilon\right\}.
\] 

\begin{lem}\label{asymp}
The sequence $(u_\epsilon)$, defined in $B_R(0)\cup\Gamma_\epsilon$ as described above, satisfies the conditions of Lemma \ref{Ascoli} in $\overline{B}_R(0)$. 
\end{lem}

\begin{proof}
By Lemma \ref{maks}, the sequence $(u_\epsilon)$ is uniformly bounded in $\overline{B}_R(0)$. Fix $\eta>0$. Since $u$ is uniformly continuous in $B_R(0)\cup \Gamma_\epsilon$, there is $r_1>0$ such that $x,y\in B_R(0)\cup \Gamma_\epsilon$, $\left|x-y\right|<r_1$, implies   
\[
\left|u(x)-u(y)\right|<\eta/2.
\]
When $x,y\in \partial B_R(0)$, the same estimate holds between $u_\epsilon(x)$ and $u_\epsilon(y)$ for all $0<\epsilon<R$, since $u_\epsilon=u$ on $\Gamma_\epsilon$.

Let us next work out the case $x\in B_R(0)$, $y\in \partial B_R(0)$. Select $0<s<S<r_1$ and $z\in \Gamma_\epsilon$ such that $y\in \partial B_s(z)$ and $B_{2S}(z)\subset B_R(0)\cup \Gamma_\epsilon$. Consider a function $v:\overline{B}_{2S}(z)\setminus B_s(z)\rightarrow \R$, 
\begin{equation*}
\begin{cases}
\Delta v=-4(n+2)\sup_\Omega f\ &\text{in}\ B_{2S}(z)\setminus \overline{B}_s(z), \\
v=\sup_{B_R (z)}u\ &\text{on}\ \partial B_s(z), \\
v=\sup_{\Gamma_\epsilon}u\ &\text{on}\ \partial B_{2S}(z).
\end{cases}
\end{equation*}
Note that this function satisfies
\[
v(a)=\vint_{B_\epsilon(a)}v dy+2\epsilon^2\sup_\Omega f
\]
when $B_\epsilon(a)\subset B_{2S}(z)\setminus \overline{B}_s(z)$. 

Let $r_2<S-s$ be so small that
\[
\sup_{B_{r_2}(y)}v<\sup_{B_S(z)}u+\eta/2.
\] 
Pick $x\in B_{r_2}(y)\cap B_R(0)$. Since $\left|x-y\right|<S-s$, by the triangle inequality $x\in B_R(0)\cap B_S(z)$. Let $\epsilon<S$. We start a game from $x_0=x$. Minnie plays with the following strategy $S^\text{Min}_{\text{II}}$: at $x_{k-1}$, she aims to a point $x_k$ where 
\[
v(x_k)\leq \inf_{B_\epsilon (x_{k-1})}v+\frac{1}{2}\epsilon^2 \sup_\Omega f.
\] 
We stop the game when $x_k\in B_{2S}(z)\setminus (B_S(z)\cap B_R(0))$ for the first time. Let this stopping time be $\tau^*$. 

Max plays with a strategy $S_{\text{I}}$. From radial convexity of $v$ we obtain 
\begin{align*}
\mathbb{E}^{x_0}_{S_\text{I},S^\text{Min}_\text{II}}&(v(x_k)|\mathcal{F}_{k-1})\\
& \leq \frac{\alpha}{2}\left\{\sup_{B_\epsilon (x_{k-1})}v+\inf_{B_\epsilon (x_{k-1})}v+\frac{1}{2}\epsilon^2 \sup_\Omega f\right\}+\beta \vint_{B_\epsilon(x_{k-1})}vdy\\
& \leq \alpha v(x_{k-1})+\frac{\alpha}{4}\epsilon^2 \sup_\Omega f+\beta v(x_{k-1})-2\epsilon^2 \sup_\Omega f\\
& \leq v(x_{k-1})-\epsilon^2 \sup_\Omega f.
\end{align*}
Hence $M_k:=v(x_k)+k\epsilon^2 \sup_\Omega f$ is a supermartingale and we obtain
\begin{align*}
u_\epsilon(x_0)&\leq \sup_{S_\text{I}}\mathbb{E}^{x_0}_{S_\text{I},S^\text{Min}_\text{II}}\left[u_\epsilon(x_{\tau^*})+\epsilon^2\sum^{\tau^*-1}_{i=0}f(x_i)\right]\\
& \leq \sup_{S_\text{I}}\mathbb{E}^{x_0}_{S_\text{I},S^\text{Min}_\text{II}}\left[v(x_{\tau^*})+\tau^*\epsilon^2 \sup_\Omega f\right]\\
& \leq v(x_0).
\end{align*}
We conclude that when $x\in B_R(0)$, $y\in \partial B_R(0)$, $\left|x-y\right|<r_2$ and $\epsilon<S$, we have
\[
\left|u_\epsilon(x)-u_\epsilon(y)\right|\leq \left|u_\epsilon(x_0)-\sup_{B_S(z)\cap \Gamma_\epsilon}u\right|+ \left|\sup_{B_S(z)\cap \Gamma_\epsilon}u-u_\epsilon(y)\right|<\eta. 
\] 
 
Finally, let us examine the case $x,y\in B_R(0)$. If $\text{dist}(x,\partial B_R(0))<r_2 /5$ and $\left|x-y\right|<r_2 /5$, there is $y_0\in \partial B_R(0)$ such that $\left|x-y_0\right|<r_2$ and $\left|y-y_0\right|<r_2$. Then 
\[
\left|u_\epsilon(x)-u_\epsilon(y)\right|<2\eta.
\]
Hence we can assume that $x,y\in B_{R-r_2/3}(0)$, and the asymptotic estimate 
\[
\left|u_\epsilon(x)-u_\epsilon(y)\right|<\eta
\]
follows straightforwardly from Theorem \ref{Lemma 1}. 
\end{proof} 

The ideas in the proof of Lemma \ref{raja} are similar to those used in the proof of \cite[Theorem 4.9]{prop}, where the uniform limit of the value functions of tug-of-war with noise was shown to be a viscosity solution to the homogeneous $p$-Laplace equation.

\begin{lem}\label{raja}
Let $u$ be a viscosity solution to \eqref{vissi} in $\Omega$ and $B_{2R}(0)\subset \Omega$. Then $u$ can be approximated uniformly by value functions of tug-of-war with noise and running payoff in $B_R(0)$. 
\end{lem}

\begin{proof}
Let $(u_\epsilon)$ be a sequence of value functions in $B_R(0)$ with final payoff $u$ and running payoff 
\[
\overline{f}=\frac{p\beta}{2(n+2)}f.
\]
By Lemma \ref{Ascoli}, it follows from Lemma \ref{asymp} that there is a uniformly continuous function $v$ in $\overline{B}_R(0)$, $v=u$ on $\partial B_R(0)$, such that there is a subsequence of $(u_\epsilon)$ converging uniformly to $v$ in $\overline{B}_R(0)$ when $\epsilon\rightarrow 0$. For convenience of notation, we denote this subsequence $(u_\epsilon)$. We are going to show that the function $v$ is a viscosity solution to \eqref{vissi} in $B_R(0)$. By comparison principle (see e.g.\ \cite[Theorem 5]{Kawohl}, and also \cite{Lu Wang}), we will conclude that $v=u$ in $B_R(0)$.

Choose a point $x\in B_R(0)$. We only work out the supersolution part, since the subsolution part is similar. Let $\phi\in C^2(B)$, $\nabla \phi(x)\neq 0$ or $D^2\phi(x)=0$, be defined in a neighborhood $B$ of $x$, touching $v$ from below in $x$. We need to show that
%\[
%\Delta^N_p \phi(x)=\frac{1}{p}\left|\nabla \phi(x)\right|^{2-p}\left((p-2)\Delta_\infty \phi(x)+\Delta \phi(x)\right)\leq f(x).
%\]
\begin{equation}\label{vissiin}
p(\Delta^N_p \phi(x)+f(x))=(p-2)\Delta^N_\infty \phi+\Delta \phi+pf(x)\leq 0.
\end{equation}
If $\nabla \phi(x)=0$ and $D^2\phi(x)=0$, we have $\Delta^N_p \phi(x)=0$ and inequality \eqref{vissiin} cannot hold. Hence we can assume $\nabla \phi(x)\neq 0$. From Taylor expansion results in \cite{asymp} it follows that there is a point $\overline{x}^\epsilon\in B_\epsilon (x)$ in the direction of $\nabla \phi(x)$ such that
\begin{align*}
& \frac{\alpha}{2}\left\{\sup_{B_\epsilon (x)}\phi+\inf_{B_\epsilon (x)}\phi\right\}+\beta \vint_{B_\epsilon (x)}\phi(y)dy-\phi(x)\\
& \geq \frac{\beta \epsilon^2}{2(n+2)}\left((p-2)\left\langle D^2 \phi(x)\left(\frac{\overline{x}^\epsilon-x}{\left|\overline{x}^\epsilon-x\right|}\right),\left(\frac{\overline{x}^\epsilon-x}{\left|\overline{x}^\epsilon-x\right|}\right)\right\rangle+\Delta\phi(x)\right) \\
&\phantom{{}=\frac{\beta \epsilon^2}{2(n+2)}} +o(\epsilon^2).\notag
\end{align*}
  
Since $v$ is the uniform limit of the sequence $(u_\epsilon)$, there is a sequence $(x_\epsilon)\subset B$ converging to $x$ so that 
\[
u_\epsilon (y)-\phi (y)\geq -\epsilon^3
\]
when $y\in B_\epsilon(x_\epsilon)$. Using the DPP characterization of $u_\epsilon$ we obtain 
\begin{align*}
\epsilon^3 \geq -\phi(x_\epsilon)&+\frac{\alpha}{2}\left\{\sup_{B_\epsilon (x)}\phi+\inf_{B_\epsilon (x)}\phi\right\}+\beta \vint_{B_\epsilon (x)}\phi(y)dy \\
& +\epsilon^2\overline{f}(x_\epsilon).
\end{align*}
Hence
\begin{align*}
-\epsilon^3 \geq &\frac{\beta \epsilon^2}{2(n+2)}\left((p-2)\left\langle D^2 \phi(x)\left(\frac{\overline{x}^\epsilon-x}{\left|\overline{x}^\epsilon-x\right|}\right),\left(\frac{\overline{x}^\epsilon-x}{\left|\overline{x}^\epsilon-x\right|}\right)\right\rangle+\Delta\phi(x)\right) \\
& +\epsilon^2\overline{f}(x_\epsilon)+o(\epsilon^2).
\end{align*}
Since $\overline{f}$ is continuous and $\nabla \phi(x)\neq 0$, dividing by $\epsilon^2$ and then letting $\epsilon\rightarrow 0$ we get
\[
0\geq \frac{\beta}{2(n+2)}\left((p-2)\Delta^N_\infty \phi(x)+\Delta \phi(x)\right)+\overline{f}(x).
\]
Remembering how the running payoff $\overline{f}$ was chosen, we have 
\[
0\geq \frac{\beta}{2(n+2)}\left((p-2)\Delta^N_\infty \phi(x)+\Delta \phi(x)+pf(x)\right).
\]
Hence $v$ is a viscosity supersolution to \eqref{vissi} in $B_R(0)$. By similar argument $v$ is also a viscosity subsolution, hence a viscosity solution. By the discussion in the beginning of the proof, the proof is complete.
\end{proof}

\begin{lause}
Nonnegative viscosity solutions of \eqref{vissi} are locally Lipschitz continuous and satisfy Harnack's inequality.   
\end{lause}

\begin{proof}
By previous lemma, each viscosity solution can be approximated locally uniformly by value functions. Hence, Harnack's inequality for viscosity solutions follows immediately from Theorem \ref{päälause}. By Theorem \ref{Lemma 1}, value functions are locally Lipschitz continuous up to the scale $\epsilon$ with a Lipschitz constant depending only on $p$ and $n$. Therefore viscosity solutions are locally Lipschitz continuous.
\end{proof}

\noindent \textbf{Acknowledgement.} The author would like to thank Mikko Parviainen for many discussions and insightful comments regarding this work.

\end{document}